\documentclass[11pt,twoside,reqno,centertags,draft]{amsart}


 \usepackage{amsmath,amsthm,amsfonts,amssymb}
\usepackage[mathscr]{euscript}
 \usepackage{hyperref} 
 \usepackage{mathrsfs} 
\usepackage{graphicx}
\usepackage{color}
  \usepackage{epsfig}

\usepackage{amssymb,amsfonts,amsthm}
\usepackage{amsmath,tikz}
\usepackage{calc,color}
\usepackage{graphics,hyperref}
\usepackage{pdfsync}
\usepackage{epstopdf}
\usepackage{enumerate}
\usepackage{bm}
\newtheorem{thm}{Theorem}[section]
\newtheorem{cor}[thm]{Corollary}
\newtheorem{lem}[thm]{Lemma}
\newtheorem{prop}[thm]{Proposition}

\theoremstyle{definition}

\theoremstyle{remark}
\newtheorem{rem}[thm]{Remark}
\numberwithin{equation}{section}
\newcommand{\ds}{\displaystyle}
\newcommand{\mc}{\mathcal}
\newcommand{\Int}{\displaystyle\int}

\newcommand{\re}{\mathbb R}
\newcommand{\na}{\mathbb N}

\newcommand{\lp}{\left(}
\newcommand{\rp}{\right)}
\newcommand{\ls}{\left[}
\newcommand{\rs}{\right]}
\newcommand{\lc}{\left\{}
\newcommand{\rc}{\right\}}
\newcommand{\eps}{\varepsilon}

\newcommand{\dd}{\,\mathrm{d}}
\newcommand{\dom}{\Omega}
\newcommand{\bnd}{\Gamma}
\newcommand{\wh}[1]{\widehat{#1}}

\newcommand{\bu}{\bm{u}}
\newcommand{\bv}{\bm{v}}
\newcommand{\bx}{\bm{x}}
\newcommand{\by}{\bm{y}}
\newcommand{\bz}{\bm{z}}

\newcommand{\vep}{\varepsilon}
\newcommand{\vph}{\varphi}

\def\loc{\operatorname{loc}}

\begin{document}
\title{Existence and regularity of minimizers for nonlocal energy functionals}
\thanks{The first author was supported by the NSF award DMS-1716790. The second author was supported by a Simons Grant and by the NSF award DMS-1716790. The third author was supported by NSF grant DMS-1211232 }
\date{}
\maketitle     
 
\vspace{ -1\baselineskip}

{\small
\begin{center}
 {\sc Mikil D. Foss, Petronela Radu, Cory Wright} \\
University of Nebraska-Lincoln, Lincoln, NE 68588 \\[10pt]
\end{center}
}

\numberwithin{equation}{section}
\allowdisplaybreaks

 \begin{quote}
\footnotesize
{\bf Abstract.}  
 In this paper we consider minimizers for nonlocal energy functionals generalizing elastic energies that are connected with the theory of peridynamics \cite{Silling2000} or nonlocal diffusion models \cite{Rossi}. We derive nonlocal versions of the Euler-Lagrange equations under two sets of growth assumptions for the integrand. Existence of minimizers is shown for integrands with joint convexity (in the function and nonlocal gradient components). By using the convolution structure we show regularity of solutions for certain Euler-Lagrange equations. No growth assumptions are needed for the existence and regularity of minimizers results, in contrast with the classical theory.

\end{quote}
\section{Introduction}

In this paper we consider minimizers for nonlocal functionals of the form
\begin{equation}\label{NLF}
\mathcal{F}(\bu):=\int_\Omega \int_\Omega f(\bx,\by-\bx,\bu(\bx),\bu(\bx)-\bu(\by))\dd \bx\dd \by,
\end{equation}
where $\bu$ satisfies (possibly free) boundary conditions. For this problem we have several goals: first, we investigate the validity of the weak form of the Euler-Lagrange equations, then we follow follow with the study of existence, uniqueness, and regularity of minimizers for some particular integrands. Of special interest to us is the connection between nonlocal problems and their classical counterparts; this aspect is studied in a general context, but also illustrated through several examples that are motivated by well known problems in classical theory (such as semilinear elliptic systems).

Nonlocal equations have risen as prime candidates to model physical phenomena characterized by singular or discontinuous behavior. As such, they have been used in population models \cite{CF}, nonlocal diffusion \cite{Rossi}, dynamic fracture \cite{Silling2000}; see also the monograph \cite{BV} for additional motivation and physical models. Whereas differential operators capture ``deep" properties at a point for a given function, integral operators incorporate a cumulative perspective. Moreover, integral operators of an appropriate structure can be used to approximate differential operators as the radius for the horizon of the interaction kernel shrinks to zero on the space of sufficiently smooth functions \cite{FR, DuMengesha}. A critical feature (which is a major advantage for considering a problem in the nonlocal framework) is the fact that the solution of a nonlocal system requires only minimal assumptions, such as measurability, monotonicity, whereas for classical systems, Sobolev differentiability is often a vital requirement.

In the nonlocal framework energy functionals have been introduced where the place of pointwise gradients has been taken by ``two-point" nonlocal gradients. These nonlocal gradients are basically weighted differences of the solution, and as such, they require just basic integrability assumptions. Despite being endowed with a reduced topological structure, the nonlocal framework has benefitted from the development of analytical tools for establishing estimates and compactness. As an example, the direct method of calculus of variations becomes applicable as coercivity is provided by a nonlocal Poincar\'e inequality and convexity serves for providing lower semi-continuity. Other results (e.g. Dirichlet's principle, existence of minimizers, weighted mean-value theorem etc. \cite{FR,HR}) have contributed to show that the nonlocal theory is a low-regularity counterpart of the classical theory. The ultimate goal is to provide a complete picture for a nonlocal calculus of variations theory as a counterpart to classical results (for which \cite{Dac} serves as a great reference).  Finally, establishing the connection between local and nonlocal frameworks would aid in correctly identifying the regime for the appropriate mathematical model to be used in applications (from low-regularity to high-regularity) as well as the correct implementation of a transition zone between these formulations.

The first part of this paper is dedicated to establishing necessary conditions under which the Euler-Lagrange equations associated with nonlocal energy functionals hold. As in the classical case (see \cite{Dac}, for example) polynomial growth conditions on the integrand are needed. For existence of minimizers, however, we show that, in contrast with the classical theory, only coercivity is needed so the integrand does not need to satisfy any upper-bound restrictions. Finally, by taking advantage of the convolution structure of the nonlocal operator (which is not smoothing), we present regularity results for solutions of nonlinear Euler-Lagrange equations under mild assumptions for the nonlinearity. Note that in the classical setting a key tool in establishing regularity is using smoothing properties of the operator, which here are lacking.

\subsection{Literature review.} Several recent works \cite{BMP1, BMP, BP, Elbau, FG, Pedregal1997} have investigated calculus of variations problems in the nonlocal setting and the relationship to the classical framework. Relevant to our investigation is the paper \cite{BellidoMora} where the authors derive Euler-Lagrange equations and prove existence of minimizers for a nonlocal functional of the form
\begin{equation}\label{EC}
\mathcal{I}(\bu):=\int_\Omega\int_\Omega w(\bx-\by,\bu(\bx)-\bu(\by))\dd \by\dd \bx- \int_\Omega F(\bx,\bu(\bx))\dd \bx.
\end{equation}
Note that the functional $\mathcal{I}$ is a particular case of the functional $\mathcal{F}$ that we considered in \eqref{NLF} since the pointwise dependence on $\bx$ is given separately through the function $F$.

On the other hand, the paper \cite{Elbau} concerns a very general nonlocal functional of the form
\begin{equation}\label{NF}
\mathcal{E}(\bu):=\int_\Omega\int_\Omega w(\bx,\by,\bu(\bx),\bu(\by))\dd \by\dd \bx.
\end{equation}
The results of \cite{Elbau} include the derivation of a condition equivalent to the lower-semicontinuity of the functional \eqref{EC}. Recall that in the classical formulation of calculus of variations quasiconvexity  provides a sine qua non condition for lower-semicontinuity, under appropriate assumptions \cite{MCoV}. Very similar to quasiconvexity, however, the condition given in \cite{Elbau} is cumbersome to verify directly. To forego this difficulty, we provide an existence of minimizers theorem obtained under joint convexity assumptions, a much stronger assumption, which in turn is very easy to verify (see Theorem \ref{jc}). More details that show the relation between our results and existing literature are provided in the next subsection.

\subsection{Main results and significance}

Our contributions in this paper concern several aspects regarding the functional $\mc{F}$ of \eqref{NLF}:
\begin{itemize}
\item We derive and establish the validity of the Euler-Lagrange equations for the nonlocal energy functional $\mc{F}$ under very weak assumptions; see Theorem \ref{NELE}. Prior to this paper, the results of \cite[Theorem 8.3]{BellidoMora} established Euler-Lagrange equations for energy functionals  of the particular form given by \eqref{EC}, where dependence on $\bu(\bx)$ is allowed through a separate term, independent of the ``nonlocal gradient" term $\bu(\bx)-\bu(\by)$. The equations are derived here in the weak form as well as in the strong (pointwise) formulation; separately, we discuss the connections between the nonlocal (integral) and classical (differential) frameworks.
\item As previously mentioned, under joint convexity assumptions, we obtain existence of a unique minimizer for $\mc{F}$ in Theorem \ref{jc}. The striking difference here in comparison to the results of \cite{BMP1,BMP, FG, Elbau} is that we do not need {\it any} upper bounds for the integrand.
\item For minimizers that satisfy the Euler-Lagrange equations we obtain the regularity results of Theorem \ref{RoS}. To our knowledge these are some of the first regularity results obtained for nonlinear nonlocal problems, and they are obtained under quite general assumptions on the nonlinearity. In the particular case of power nonlinearities, no condition on the subcriticality of exponent is necessary, contrasting the classical case where the Sobolev embedding theorem dictates the regime of exponents for which regularity results hold.
\item Along the way (and also separately in section \ref{Ex}) we provide examples for which the general results apply, and also establish connections between the nonlocal and local theories.

\end{itemize}

As an additional remark, note that the functionals considered here do not fit the framework of \cite{Elbau} since we do not have pairwise-symmetry of the integrand.

\subsection{Organization of the paper} In the next section we introduce some standard (by now) nonlocal operators as well as some notation and spaces.  In section \ref{ELS}, we find appropriate growth conditions under which we derive the Euler-Lagrange equations for our nonlocal functional. Using a specific form of the functional, we also show some connections between the classical and nonlocal formulations.  In section \ref{EoM}, we find sufficient conditions for existence of minimizers of the nonlocal functional \eqref{NLF}.  In section \ref{RoM}, we establish regularity results for a specific Euler-Lagrange equation, $\mathcal{L}_{\bm{\mu}}[\bu](\bx)=f(\bx,\bu),$ where $\mc{L}_{\bm{\mu}}$ is a nonlocal Laplacian that will be introduced below.  Finally, in section \ref{Ex}, we take specific examples of functionals and apply our previous results to establish existence and regularity to certain nonlocal equations.

\section{Preliminaries}\label{prelim}

In this section we introduce notation that will be used throughout the paper. Let $\Omega, \Gamma \subseteq \re^n$ be open bounded domains with $\Gamma$ a collar around $\Omega$. Let $\bu=(u_1, u_2, ...., u_N), \bv=(v_1, v_2, ...., v_N)$ denote two general vectors in $\re^N$.
\begin{itemize}
\item Denote by $\otimes$ the standard tensor product with components
\[
(\bu \otimes \bv)_{ij} =u_iv_j, \quad i,j=1...N.
\]
\item For a set $A \subseteq \re^n$, denote
\begin{equation}\label{hatset}
\widehat{A}:=\{\by-\bx:\bx,\by \in A\}.
\end{equation}
\item For a function $\bm{\psi}:A \to \re^N$, $\widehat{\bm{\psi}}:A^2 \to \re^N$ is given by
\[
\widehat{\bm{\psi}}(\bx,\by)=\bm{\psi}(\by)-\bm{\psi}(\bx).
\]
\item If $\beta:(\Omega \cup \Gamma)\times \widehat{\Omega \cup \Gamma} \to \re$ is measurable, then define
\[
    L_{\beta}^p((\Omega \cup \Gamma)^2)\!
    :=\!
    \lc\bm{w}\!\!:\!(\Omega \cup \Gamma)^2 \to \mathbb{R}^N\,\!\!\!:\!\!\!\,
    \int_{(\Omega \cup \Gamma)^2} \!\!\!\!\!\!|\beta(\bx,\by-\bx)||\bm{w}(\bx,\by)|^{p}
        \dd \bx\dd \by\!<\!\infty\rc
\]
with the norm
\[
\|u\|^p_{L_{\beta}^p((\Omega \cup \Gamma)^2)}:=\int_{(\Omega \cup \Gamma)^2} |\beta(\bx,\by-\bx)||\bm{w}(\bx,\by)|^{p}
        \dd \bx\dd \by.
\]
\item Define
\[
    \mathcal{W}_{\beta}^p(\Omega\cup\Gamma)
    :=
    \lc\bu\in L^p(\Omega \cup \Gamma;\re^N)\,:\,
    \wh{\bu}(\bx,\by) \in L_{\beta}^p((\Omega \cup \Gamma)^2)\rc
\]
 and
 \[
    \mathcal{W}_{\beta,\Gamma'}^p(\Omega \cup \Gamma)
    :=
    \lc\bu\in \mc{W}_{\beta}^p(\Omega \cup \Gamma)\,:\,
    \bu=0 \text{ on } \Gamma\setminus\Gamma'\rc,
 \]
 where $\Gamma'\subset \Gamma$ (see Figure 1 below). Whenever $\Gamma'=\emptyset$ we have that $\bu=0 \text{ on } \Gamma$ and we write $\mathcal{W}_{\beta,\Gamma'}^p(\Omega \cup \Gamma)=:\mc{W}_{\beta,0}^p(\Omega \cup \Gamma)$.
\end{itemize}
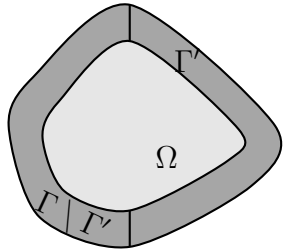
\begin{figure}[h]
\definecolor{grey1}{gray}{.9}
\definecolor{grey2}{gray}{.65}
\begin{tikzpicture}[scale=0.7]
 \draw[thick,scale=2.3,fill=grey2] plot [smooth cycle] coordinates {(1,0) (.25,.25) (0,1) (.2,1.4) (1,2)  (2.2,1) (2,0.5) };
  \draw[thick,scale=2.3,fill=grey1] plot [smooth cycle] coordinates {(1,.3) (.46,.46) (.28,.9) (.412,1.198) (1,1.7)  (1.9,.97) (1.81,0.712) };
  \node at (3,1.7) {\large$\Omega$};
  \node [rotate=-25] at (1.2,.55) {\large$\Gamma\setminus\Gamma'$};
  \node at (3.4,3.55) {\large$\Gamma'$};
  \draw [thick, scale=2.3](1,0) -- (1,.3);
  \draw [thick,scale=2.3] (1,2) -- (1,1.7);
\end{tikzpicture}
\caption{Domain $\Omega$ surrounded by a collar $\Gamma$ with boundary conditions imposed on $\Gamma\setminus \Gamma'$}
\end{figure}

Also, for notational convenience, given a map $(\bm{x},\bm{y})\mapsto\bm{g}(\bm{x},\bm{y}):\re^{n_1}\times\re^{n_2}\to\re^N$, the matrix (in $\re^{N\times n_1}$) of all first-order partial derivatives of $\bm{g}$, with respect to the first $n_1$ components, will be denoted by $\partial_{\bm{x}}\bm{g}(\bm{x},\bm{y})$ or just $\partial_{\bm{x}}\bm{g}$.

\begin{rem}\label{Lpb}
Note $L_{\beta}^p$ is the space of $p$-integrable functions with respect to the measure $\beta(\bx,\by-\bx)\dd \bx\dd \by$.  Thus, $\mathcal{W}^p_{\beta,\Gamma'}$ is a Banach space.
\end{rem}

 For $\bm{\alpha}:\Omega\cup\Gamma \times \Omega\cup\Gamma\rightarrow\re^N, \,\, \bu:\Omega\cup\Gamma\to\re^N,$ and $\bm{F}:\Omega\cup\Gamma \times \Omega\cup\Gamma\rightarrow\re^{N\times N}$ we define as in \cite{DGLz2013} the following generalized nonlocal operators:\\
\begin{enumerate}[(i)]
\item Nonlocal gradient
\begin{equation}\label{G}
\mathcal{G}_{\bm{\alpha}}[\bu](\bx,\by):=(\bu(\by)-\bu(\bx))\otimes\bm{\alpha}(\bx,\by), \quad \bx,\by \in \Omega.
\end{equation}
\item Nonlocal divergence
\begin{equation}\label{D}
 \mathcal{D}_{\bm{\alpha}}[\bm{F}](\bx):=\Int_{\Omega\cup\Gamma}(\bm{F}(\bx,\by)\bm{\alpha}(\bx,\by)-\bm{F}(\by,\bx)\bm{\alpha}(\by,\bx))\dd \by,\quad \bx \in \Omega.
 \end{equation}
\item Nonlocal Laplacian, for $\bm\alpha$ symmetric
\begin{equation}\label{L}
\mathcal{L}_{\bm{\alpha}^2}[\bu](\bx):=\mathcal{D}_{\bm{\alpha}}[\mathcal{G}_{\bm{\alpha}}[\bu]](\bx)=2\int_{\Omega \cup \Gamma} \widehat{\bu}(\bx,\by)|\bm\alpha(\bx,\by)|^2 \dd \by.
\end{equation}
\item Nonlocal $p$-Laplacian, for $\bm\alpha$ symmetric
\begin{align}\label{pL}
\begin{aligned}
\mathcal{L}_{\bm{\alpha}}^p[\bu](\bx)&:=\mc{D}_{\bm\alpha}[|\mc{G}_{\bm\alpha}[\bu]|^{p-1}\mc{G}_{\bm\alpha}[\bu]](\bx)\\
&=2\int_{\Omega \cup \Gamma}|\widehat{\bu}(\bx,\by)|^{p-2}\widehat{\bu}(\bx,\by)|\bm\alpha(\bx,\by)|^{p}\dd \by.
\end{aligned}
\end{align}

\end{enumerate}

\section{Euler-Lagrange Equations}\label{ELS}

In this section we aim to establish a nonlocal analogue to the Euler-Lagrange equations of the functional
\[
\mathcal{F}_{\loc}[\bu]=\int_{\Omega}f(\bx,\bu,\nabla \bu)\dd \bx.
\]
To this end, throughout this section, let $f:\dom\cup\bnd\times\widehat{\dom\cup\bnd}\times\re^N\times\re^N$ be given. We consider the functional $\mathcal{F}:L^1(\dom\cup\bnd;\re^N)$ defined by
\begin{equation}\label{Fu}
\mathcal{F}[\bu]:=\Int_{\Omega\cup\Gamma}\Int_{\Omega\cup\Gamma}f(\bx,\by-\bx,\bu(\bx), \bu(\by)-\bu(\bx))\dd \by\dd \bx.
\end{equation}

Before stating the result, we introduce our main assumptions for the integrand of  $\mathcal{F}$, which include two variations of a growth condition.  The underlying ideas for the proof of the necessity for the nonlocal Euler-Lagrange equations are similar to those presented in \cite{Dac} for local problems.

The growth conditions below are the analogues, for the nonlocal setting, of controllable growth conditions for local variational problems. Unlike the local setting, we do not consider controllable growth conditions where the growth with respect to $\bu$ is faster than the growth with respect to $\bz$. For local problems, a faster growth with respect to $\bu$ can be managed by taking advantage of the Sobolev embedding theorem, but there is no such embedding available in the nonlocal setting unless the kernel $\beta$ is strongly singular.


\begin{itemize}
\item[(R)]{\it Regularity Condition}:
\begin{itemize}
    \item[$\bullet$] $(\bx,\bz)\mapsto f(\bx,\bz,\bu,\bm{\xi})$ is measurable, for each $\bu,\bm{\xi}\in\re^N$;
    \item[$\bullet$] $(\bu,\bm{\xi})\mapsto f(\bx,\bz,\bu,\bm{\xi})$ is continuously differentiable, for \\
    a.e. $(\bx,\bz)\in\dom\cup\bnd\times\widehat{\dom\cup\bnd}$.
\end{itemize}
\item[(GI)] {\it Growth Condition I}: For each $R \ge 0$ there exists $a_R\in L^1(\Omega \cup \Gamma\times\widehat{\Omega \cup \Gamma})$ such that for every $\bu,\bm{\xi}\in\re^N$ satisfying $|\bu|,|\bm{\xi}| \le R$ we have
\[
    |\partial_{\bu}f(\bx,\by-\bx,\bu,\bm{\xi})|,
    |\partial_{\bm{\xi}} f(\bx,\by-\bx,\bu,\bm{\xi})|
    \le  a_R(\bx,\by-\bx),
\]
for a.e. $(\bx,\by)\in \left[\Omega \cup \Gamma\right]^2$.\\
\item[(GII)] {\it Growth Condition II}: With $p\in[1,\infty)$ given, there exist measurable $a:\Omega \cup \Gamma\times\widehat{\dom\cup\bnd}\to[0,\infty)$ and $\beta:\dom\cup\bnd\times\widehat{\dom\cup\bnd}\to\re^N$ such that
\[
    \bx\mapsto\int_{\Omega\cup\Gamma}
    a(\bx,\by-\bx)\dd\by\in L^{\frac{p}{p-1}}(\dom\cup\bnd),
\]
\[
    \bx \mapsto \int_{\Omega \cup \Gamma}
    |\beta(\bx,\by-\bx)|\dd\by\in L^\infty(\dom\cup\bnd)
\]
and
\begin{align*}
    |\partial_{\bu}f(\bx,\by-\bx,\bu,\bm{\xi})|
    &\le
    a(\bx,\by-\bx)+|\beta(\bx,\by-\bx)|(|\bu|^{p-1}+|\bm{\xi}|^{p-1})\\
    |\partial_{\bm{\xi}} f(\bx,\by-\bx,\bu,\bm{\xi})|
    &\le
    a(\bx,\by-\bx)+|\beta(\bx,\by-\bx)|(|\bu|^{p-1}+|\bm{\xi}|^{p-1})
\end{align*}
for a.e. $(\bx,\by) \in\left[\Omega \cup \Gamma\right]^2$ and every $(\bu,\bm{\xi}) \in \re^N \times \re^N$.
\end{itemize}

\begin{thm}[Nonlocal Euler-Lagrange equations]\label{NELE}
Let $p\in[1,\infty)$, a map $\bu_0\in\mathcal{W}^p_{\beta}(\dom\cup\bnd;\re^N)$ and a collar subset $\bnd'\subseteq\bnd$, be given. Define the admissible class
\[
    \mathcal{A}:=\left\{\bm{v}\in \mc{W}^p_{\beta}(\dom\cup\bnd;\re^N)\,:\,
        [\bm{v}-\bu_0]\in\mathcal{W}^p_{\beta,\bnd'}(\dom\cup\bnd;\re^N)\right\}
\]
Suppose that
\begin{itemize}
    \item $f$ has regularity (R) above
    \item $\bu\in\mathcal{A}$ is a minimizer for $\mathcal{F}$, defined in \eqref{Fu}, over $\mathcal{A}$.
\end{itemize}
 Then, if either
\begin{itemize}
\item[{\it (i)}]
$f$ also satisfies (GI) and $\bu \in \mathcal{A}\cap L^\infty$, in which case the set $\mathcal{V}:=\mathcal{W}^p_{\beta,\bnd'}\cap L^\infty$;
\end{itemize}
or
\begin{itemize}
\item[{\it(ii)}]
$f$ also satisfies (GII), in which case the set $\mathcal{V}:=\mathcal{W}^p_{\beta,\bnd'}(\Omega \cup \Gamma;\re^N)$;
\end{itemize}
then for all $\bm{\varphi}\in\mathcal{V}$, we find that
\begin{equation}\label{ELe2}
\int_{\Omega \cup \Gamma'}\int_{\Omega \cup \Gamma}\bm{\varphi}(\bx)\cdot\partial_{\bu}f(\bx,\by-\bx,\bu,\widehat{\bu})
-\widehat{\bm{\varphi}}(\bx,\by)\cdot\partial_{\bm{\xi}}f(\bx,\by-\bx,\bu,\widehat{\bu}) \dd \by\dd \bx=0,
\end{equation}
or, equivalently,
\begin{multline}\label{ELe}
    \int_{\Omega\cup\Gamma'}\bm{\varphi}(\bx)\cdot\int_{\Omega \cup \Gamma}
        \partial_{\bu}f(\bx,\by-\bx,\bu,\widehat{\bu})\\
            -\left[\partial_{\bm{\xi}}f(\bx,\by-\bx,\bu,\widehat{\bu})-
            \widetilde{\partial_{\bm{\xi}}f}(\bx,\bx-\by,\bu,\widehat{\bu})\right]
    \dd\by\dd\bx=0.
\end{multline}
Here
\[
    \widetilde{\partial_{\bm{\xi}}f}(\bx,\bm{z},\bu(\bx),\widehat{\bu}(\bx,\by)):=\partial_{\bm{\xi}} f(\by,\bm{z},\bu(\by),\widehat{\bu}(\by,\bx)),
 \]
   for each $(\bx,\bz)\in\dom\cup\bnd\times\widehat{\dom\cup\bnd}$.
\end{thm}

\begin{proof}
With $\mathcal{V}$ the appropriate space of variations, depending on whether (i) or (ii) is satisfied, let $\bm{\varphi}\in\mathcal{V}$ be given. We will verify that the G\^ateaux derivative of \eqref{Fu} in the direction of $\bm{\varphi}$ exists.  For convenience, define $h:\dom\cup\bnd\times\dom\cup\bnd\times\re\to\re$ by
\begin{multline*}
    h(\bx,\by,\vep):=
    \int_0^1\Big\{
        \bm{\vph}(\bx)\cdot\partial_{\bu}f(\bx,\by-\bx,\bu+t\vep\bm{\vph},
            \widehat{\bu}+t\vep\widehat{\bm{\vph}})\\
        +\widehat{\bm{\vph}}(\bx,\by)
            \cdot\partial_{\bm{\xi}}f(\bx,\by-\bx,\bu+t\vep\bm{\vph},
            \widehat{\bu}+t\vep\widehat{\bm{\vph}})
    \Big\}\dd t
\end{multline*}
\begin{align*}
g_{\bu}(\bx,\by)&:=\partial_{\bu}f(\bx,\by-\bx,\bu(\bx)+t\eps \bm{\varphi}(\bx),\widehat{\bu}(\bx,\by)+t\eps\widehat{\bm{\varphi}}(\bx,\by));\\
g_{\bm{\xi}}(\bx,\by)&:=\partial_{\bm{\xi}} f(\bx,\by-\bx,\bu(\bx)+t\eps \bm{\varphi}(\bx),\widehat{\bu}(\bx,\by)+t\eps\widehat{\bm{\varphi}}(\bx,\by)).
\intertext{Thus}
    h(\bx,\by,\eps)&
    =\int_0^1
    \bm{\varphi}(\bx)\cdot g_{\bu}(\bx,\by)
    +\widehat{\bm{\varphi}}(\bx,\by)\cdot g_{\bm{\xi}}(\bx,\by) \dd t.
\end{align*}
Then we may write
\begin{align}
\label{preGat}
    &\frac{\mathcal{F}[\bu+\eps \bm{\varphi}]-\mathcal{F}[\bu]}{\eps}\\
\nonumber
    &=\frac{1}{\eps}\int_{\Omega \cup \Gamma}\int_{\Omega \cup \Gamma}
    \int_0^1 \frac{d}{dt}f(\bx,\by-\bx,\bu+\eps t
        \bm{\varphi},\widehat{\bu}+t\eps\widehat{\bm{\varphi}})
        \dd t\dd \by\dd \bx\\
\nonumber
    &=\int_{\Omega \cup \Gamma}\int_{\Omega \cup \Gamma}
        h(\bx,\by,\eps) \dd \by\dd \bx.
\end{align}
Since the derivatives of $f$ with respect to $\bm{u}$ and $\bm{\xi}$ are continuous, we have
\begin{multline*}
    \lim_{\vep\to0}h(\bx,\by,\vep)
    =\bm{\vph}(\bx)\cdot\partial_{\bu}f(\bx,\by-\bx,\bu,\wh{\bu})\\
        +\wh{\bm{\vph}}(\bx)\cdot\partial_{\bm{\xi}}f(\bx,\by-\bx,\bu,\wh{\bu}),
        \quad a.e. \ (\bx,\by)\in(\Omega \cup \Gamma)^2.
\end{multline*}
We will show that under either of the growth conditions, we have
\begin{equation}\label{DCT}
|h(\bx,\by,\eps)| \le \gamma(\bx,\by),  \ a.e. \  (\bx,\by,\eps) \in (\Omega \cup \Gamma)^2 \times [-1,1]
\end{equation}
for some $\gamma \in L^1((\Omega \cup \Gamma)^2)$. From here the existence of the G\^ateaux derivative of $\mathcal{F}$ follows from the Lebesgue dominated convergence theorem.\\

\emph{Growth Condition I}: Fix $(t,\vep)\in[0,1]\times[-1,1]$. For this part of the theorem, we are assuming that $\bu, \bm{\varphi} \in L^\infty$. Thus there exists an $R>0$ such that
\begin{align*}
    & |\bu(\bx)+t\eps \bm{\varphi}(\bx)|,\, |\bm{\varphi}| \le R,
        && \text{a.e.}  \ \bx \in \Omega \cup \Gamma,\\
    & |\widehat{\bu}(\bx,\by)+t\eps\widehat{\bm{\varphi}}(\bx,\by)| \le R,
        && \text{a.e.}  \ (\bx,\by) \in (\Omega \cup \Gamma)^2.
\end{align*}
The growth assumptions for $\bm{f}_{\bm{u}}$ and $\bm{f}_{\bm{\xi}}$ imply
\begin{align}\label{b1}
\begin{aligned}
    & |g_{\bu}(\bx,\by)||\bm{\varphi}(\bx)|
        \le a_R(\bx,\by-\bx)|\bm{\varphi}(\bx)|,\\
    & |g_{\bm{\xi}}(\bx,\by)||\widehat{\bm{\varphi}}(\bx,\by)|
        \le a_R(\bx,\by-\bx)|\widehat{\bm{\varphi}}(\bx,\by)|
\end{aligned}
\end{align}
for a.e. $\bx,\by \in\Omega \cup \Gamma$.  Since $\bu,\bm{\varphi} \in  L^{\infty}(\Omega \cup \Gamma;\re^N)$, H\"older's inequality gives
\begin{align}\label{b2}
\begin{aligned}
(\bx,\by) \mapsto &a_R(\bx,\by-\bx)|\bm{\varphi}(\bx)|\in L^{1}((\Omega \cup \Gamma)^2),  \\
(\bx,\by) \mapsto &a_R(\bx,\by-\bx)|\widehat{\bm{\varphi}}(\bx,\by)| \in L^1((\Omega \cup \Gamma)^2).
\end{aligned}
\end{align}
Since $(t,\vep)\in[0,1]\times[-1,1]$ were arbitrary, \eqref{b1} and \eqref{b2} yield \eqref{DCT}, and the existence of the G\^{a}teaux derivative is proved.\\

\emph{Growth Condition II}: Again, fix $(t,\vep)\in[0,1]\times[-1,1]$. For convenience, define $\zeta:\re\times\re\to\re$ by
\[
    \zeta(\bm{u},\bm{\xi}):=|\bm{u}+t\vep\bm{\xi}|^{p-1}.
\]
Under (GII), we have
\begin{align}\label{inG2}
\begin{aligned}
    &|g_{\bu}(\bx,\by)||\bm{\varphi}(\bx)|\\
    &\qquad
    \le
        \left\{a(\bx,\by-\bx)
        +|\beta(\bx,\by-\bx)|
            \left[\zeta(\bm{u}(\bx),\bm{\varphi}(\bx))\right.\right.\\
               &\left.\left. \qquad \quad +\zeta(\widehat{\bm{u}}(\bx,\by),\widehat{\bm{\varphi}}(\bx,\by))
            \right]\right\}
                |\bm{\varphi}(\bx)|,\\
    &|g_{\bm{\xi}}(\bx,\by)||\wh{\bm{\varphi}}(\bx),\by|\\
    &\qquad
    \le
        \left\{a(\bx,\by-\bx)
        +|\beta(\bx,\by-\bx)|
            \left[\zeta(\bm{u}(\bx),\bm{\varphi}(\bx))\right.\right.\\
                  &\left.\left. \qquad \quad +\zeta(\widehat{\bm{u}}(\bx,\by),\widehat{\bm{\varphi}}(\bx,\by))
            \right]\right\}
                |\wh{\bm{\varphi}}(\bx,\by)|.
\end{aligned}
\end{align}
To establish \eqref{DCT}, it is sufficient to argue that the upper bounds in \eqref{inG2} are pointwise uniformly bounded by an $L^1$ function, independent of $(t,\vep)\in[0,1]\times[-1,1]$ over $(\dom\cup\bnd)^2$.

Using Young's inequality, we see that
\[
    \zeta(\bm{u},\bm{\xi})|\bm{\xi}|
    \le 2^{p-1}\left(|\bm{u}|^{p-1}+|\bm{\xi}|^{p-1}\right)|\bm{\xi}|
    \le C\left(|\bm{u}|^p+|\bm{\xi}|^p\right).
\]
Thus
\begin{multline*}
    \left\{a(\bx,\by-\bx)
        +|\beta(\bx,\by-\bx)|
            \left[\zeta(\bm{u}(\bx),\bm{\varphi}(\bx))
                +\zeta(\widehat{\bm{u}}(\bx,\by),\widehat{\bm{\varphi}}(\bx,\by))
            \right]\right\}
                |\bm{\varphi}(\bx)|\\
    \le
    a(\bx,\by-\bx)|\bm{\varphi}(\bx)|
    +C|\beta(\bx,\by-\bx)|\left(|\bm{u}(\bx)|^p+|\bm{\varphi}(\bx)|^p\right)\\
    +C|\beta(\bx,\by-\bx)|\left(|\widehat{\bm{u}}(\bx,\by)|^{p-1}
        +|\widehat{\bm{\varphi}}(\bx,\by)|^{p-1}\right)|\bm{\varphi}(\bx)|
\end{multline*}
and
\begin{multline*}
    \hspace{-.3cm}\left\{a(\bx,\by-\bx)
        +|\beta(\bx,\by-\bx)|
            \left[\zeta(\bm{u}(\bx),\bm{\varphi}(\bx))
                +\zeta(\widehat{\bm{u}}(\bx,\by),\widehat{\bm{\varphi}}(\bx,\by))
            \right]\right\}
                |\wh{\bm{\varphi}}(\bx,\by)|\\
    \le
    a(\bx,\by-\bx)|\bm{\varphi}(\bx)|
    +C|\beta(\bx,\by-\bx)|\left(|\bm{u}(\bx)|^{p-1}+|\bm{\varphi}(\bx)|^{p-1}\right)
        |\widehat{\bm{\varphi}}(\bx,\by)|\\
    +C|\beta(\bx,\by-\bx)|\left(|\widehat{\bm{u}}(\bx,\by)|^p
        +|\widehat{\bm{\varphi}}(\bx,\by)|^p\right)
\end{multline*}
These upper bounds are independent of both $t$ and $\vep$, so it remains to show that they both belong to $L^1$.

Recall for part (ii) of the theorem, we assume that $\bu,\bm{\vph}\in W^{p}_{\beta}$. Also, we may put
\[
    A:=\left\|\int_{\Omega \cup \Gamma}
   \! a(\cdot,\by-\cdot)|\dd\by\right\|_{L^{\frac{p}{p-1}}(\dom\cup\bnd)}
    \:\!\text{and }\:
    B:=\left\|\int_{\Omega \cup \Gamma}
    \!|\beta(\cdot,\by-\cdot)| \dd\by\right\|_{L^\infty(\dom\cup\bnd)}\!,
\]
with both $A,B<\infty$. We further define
\begin{align*}
    I_1
    &:=
    \int_{(\dom\cup\bnd)^2}
        a(\bx,\by-\bx)|\bm{\vph}(\bx)|\dd\bx\dd\by,\\
    I_2
    &:=
    \int_{(\dom\cup\bnd)^2}
        |\beta(\bx,\by-\bx)\left(|\bm{u}(\bx)|^p+|\bm{\varphi}(\bx)|^p\right)
        \dd\bx\dd\by,\\
    I_3
    &:=
    \int_{(\dom\cup\bnd)^2}
        |\beta(\bx,\by-\bx)|\left(|\wh{\bm{u}}(\bx,\by)|^{p-1}
            +|\wh{\bm{\vph}}(\bx,\by)|^{p-1}\right)
        |\bm{\vph}(\bx)|\dd\bx\dd\by\\
    I_4
    &:=
    \int_{(\dom\cup\bnd)^2}
        |\beta(\bx,\by-\bx)|
        \left(|\bm{u}(\bx)|^{p-1}+|\bm{\vph}(\bx)|^{p-1}\right)
        |\wh{\bm{\vph}}(\bx,\by)|
        \dd\bx\dd\by,\\
\intertext{and}
     I_5
    &:=
    \int_{(\dom\cup\bnd)^2}
        |\beta(\bx,\by-\bx)|
        \left(|\wh{\bm{u}}(\bx,\by)|^p+|\wh{\bm{\varphi}}(\bx,\by)|^p\right)\dd\bx\dd\by
\end{align*}

We will verify that $I_1,\dots,I_5$ are each bounded. For $I_1$ and $I_2$, Fubini's theorem and then H\"older's inequality provides
\[
    I_1=\int_{\dom\cup\bnd}|\bm{\vph(\bx)}|
        \ls\int_{\dom\cup\bnd}a(\bx,\by-\bx)\dd\by\rs\dd\bx
    \le A\|\bm{\vph}\|_{L^p}
\]
and similarly
\begin{align*}
    I_2
    &=
    \int_{\dom\cup\bnd}
        \left(|\bm{u}(\bx)|^p+|\bm{\varphi}(\bx)|^p\right)
        \ls\int_{\dom\cup\bnd}|\beta(\bx,\by-\bx)|\dd\by\rs\dd\bx\\
    &\le
    CB\lp\|\bu\|_{L^p}^p
        +\|\bm{\vph}\|_{L^p}^p\rp,
\end{align*}
where $C<\infty$ is independent of both $t$ and $\vep$. To bound $I_3$, we first reorganize the integrand and then apply Fubini's theorem and H\"{o}lder's inequality (twice):
\begin{align*}
    I_3
    &=
    \int_{(\dom\cup\bnd)^2}
        \ls|\beta(\bx,\by-\bx)|^{\frac{1}{p}}|\bm{\vph}(\bx)|\rs\\
    &\hspace*{1in}\times
        \ls|\beta(\bx,\by-\bx)|^{\frac{p-1}{p}}
            \left(|\wh{\bm{u}}(\bx,\by)|^{p-1}
            +|\wh{\bm{\vph}}(\bx,\by)|^{p-1}\right)\rs\dd\by\dd\bx\\
    &\le
    C\int_{\dom\cup\bnd}
        |\bm{\vph}(\bx)|
        \ls\int_{\dom\cup\bnd}|\beta(\bx,\by-\bx)|\dd\by\rs^{\frac{1}{p}}\\
    &\hspace*{1in}\times
        \ls\int_{\dom\cup\bnd}|\beta(\bx,\by-\bx)|
            \left(|\wh{\bm{u}}(\bx,\by)|^p
            +|\wh{\bm{\vph}}(\bx,\by)|^p\right)\dd\by
        \rs^{\frac{p-1}{p}}\dd\bx\\
    &\le
    C B\|\bm{\vph}\|_{L^p}
        \ls\int_{(\dom\cup\bnd)^2}|\beta(\bx,\by-\bx)|
            \lp|\wh{\bu}(\bx,\by)|^p+|\wh{\bm{\vph}}(\bx,\by)|^p\rp\dd\by\dd\bx
        \rs^{\frac{p-1}{p}}.
\end{align*}
This last quantity is uniformly bounded, with respect to $(t,\vep)\in[0,1]\times[-1,1]$, since $\bu,\bm{\vph}\in\mathcal{W}^p_{\beta}\subseteq L^p_{\beta}$. Thus $I_1$, $I_2$, and $I_3$ are each uniformly bounded. The argument for $I_4$ is similar to the one used for $I_3$. The bound for $I_5$ is immediate from the assumption that $\bu,\bm{\vph}\in\mathcal{W}^p_{\beta}$. Hence $I_1,\dots, I_5$ are all bounded

Returning to \eqref{inG2}, we conclude that there is a $\gamma\in L^1((\dom\cup\bnd)^2)$ such that \eqref{DCT} holds. Thus the G\^{a}teaux derivative of $\mathcal{F}$ also exists under the growth assumption (GII). Passing to the limit as $\vep\to0^+$ and then as $\vep\to0^-$ and using the fact that $\bu$ is a minimizer for $\mathcal{F}$ over $\mathcal{A}$, we obtain \eqref{ELe2}. By expanding the second integrand in \eqref{ELe2}, using the property that $\bm{\vph}\in\mathcal{V}$ implies $\bm{\vph}=\bm{0}$ on $\bnd\setminus\bnd'$, and changing the order of integration, we can produce \eqref{ELe} as follows:
\begin{align*}
    0&\!=\!\int_{\dom\cup\bnd'}\!\!\int_{\dom\cup\bnd} \!\!\!\bm{\varphi}(\bx)\cdot\partial_{\bu}f(\bx,\by-\bx,\bu,\widehat{\bu})-
        \wh{\bm{\vph}}(\bx,\by)
        \cdot\partial_{\bm{\xi}}f(\bx,\by-\bx,\bu,\wh{\bu})\dd\by\dd\bx\\
    &=
    \int_{\dom\cup\bnd'}\int_{\dom\cup\bnd}  \bm{\varphi}(\bx)\cdot\partial_{\bu}f(\bx,\by-\bx,\bu,\widehat{\bu}) -
        \bm{\vph}(\by)
        \cdot\partial_{\bm{\xi}}f(\bx,\by-\bx,\bu,\wh{\bu})\dd\by\dd\bx\\
    & \hspace*{.8in}
    +\int_{\dom\cup\bnd'}\int_{\dom\cup\bnd}
        \bm{\vph}(\bx)
        \cdot\partial_{\bm{\xi}}f(\bx,\by-\bx,\bu,\wh{\bu})\dd\by\dd\bx\\
    &=
    \int_{\dom\cup\bnd'}  \!\!\!\bm{\varphi}(\bx)\cdot\partial_{\bu}f(\bx,\by-\bx,\bu,\widehat{\bu}) -  \bm{\vph}(\by)\cdot
        \ls\int_{\dom\cup\bnd'}
           \!\!\! \partial_{\bm{\xi}}f(\bx,\by-\bx,\bu,\wh{\bu})\dd\bx
        \rs\!\!\!\dd\by\\
    &\hspace*{.8in}
    -\int_{\dom\cup\bnd'}\bm{\vph}(\bx)\cdot
        \ls\int_{\dom\cup\bnd}
        \partial_{\bm{\xi}}f(\bx,\by-\bx,\bu,\wh{\bu})\dd\by
        \rs\dd\bx
        \end{align*}
        \begin{align*}
        &=\int_{\Omega\cup\Gamma'}\bm{\varphi}(\bx)\cdot\int_{\Omega \cup \Gamma}
        \partial_{\bu}f(\bx,\by-\bx,\bu,\widehat{\bu})\\
           &\hspace*{.8in}
   -\left[\partial_{\bm{\xi}}f(\bx,\by-\bx,\bu,\widehat{\bu})-
            \widetilde{\partial_{\bm{\xi}}f}(\bx,\bx-\by,\bu,\widehat{\bu})\right]
    \dd\by\dd\bx.
\end{align*}

\end{proof}

\begin{cor}\label{sEL}[Strong nonlocal Euler-Lagrange equations]
If $f$ satisfies Theorem 3.1
then we obtain the strong form of the Euler-Lagrange equations as given by the system of integral equations
\[
\int_{\Omega \cup\Gamma} f_{\bu}(\bx,\by)-(f_{\bm{\xi}}(\bx,\by)-\tilde{f}_{\bm{\xi}}(\bx,\by))\dd \by=0, \quad \forall \bx \in \Omega \cup\Gamma'.
\]
\end{cor}

\begin{thm} If $(\bu,\bm{\xi}) \to f(\bx,\bz,\bu,\bm{\xi})$ is convex and $\overline{\bu} \in \bu_0+ \mathcal{W}^p_{\beta,0}(\Omega \cup \Gamma)$ satisfies the weak form of the Euler-Lagrange equation \eqref{ELe2} or \eqref{ELe}, then $\overline{\bu}$ is a minimizer for \eqref{Fu} with respect to the admissible class $\bu_0+\mathcal{W}^p_{\beta,0}(\Omega \cup \Gamma)$.
\end{thm}

\begin{proof}
For $\bv \in \bu_0+\mathcal{W}^p_{\beta}(\Omega \cup \Gamma)$ let $\bm{\varphi}:=\overline{\bu}-\bv \in \mathcal{W}^p_{\beta,0}(\Omega \cup \Gamma)$ (see Remark \ref{Lpb}), so $\eqref{ELe2}$ yields
\begin{align*}
\int_{\Omega \cup \Gamma}\int_{\Omega \cup \Gamma}&\bm{\varphi}(\bx) f_{\bu}(\bx,\by-\bx,\overline{\bu}(\bx),\overline{\bu}(\bx)-\overline{\bu}(\by))\\
&+\wh{\bm{\varphi}}(\bx,\by) f_{\bm{\xi}}(\bx,\by-\bx,\overline{\bu}(\bx),\overline{\bu}(\bx)-\overline{\bu}(\by)) \dd \by\dd \bx=0
\end{align*}
for the function $v$.  Then by the convexity assumption on $f$,

\begin{align*}
F[\overline{\bu}] -F[\bv]&\le \int_{\Omega \cup \Gamma}\int_{\Omega \cup \Gamma}\bm{\varphi}(\bx) f_{\bu}(\bx,\by-\bx,\overline{\bu}(\bx),\overline{\bu}(\bx)-\overline{\bu}(\by))\dd \by\dd \bx\\
&+ \int_{\Omega \cup \Gamma}\int_{\Omega \cup \Gamma}\wh{\bm{\varphi}}(\bx,\by)f_{\bm{\xi}}(\bx,\by-\bx,\overline{\bu}(\bx),\overline{\bu}(\bx)-\overline{\bu}(\by)) \dd \by\dd \bx\\
&=0.
\end{align*}
Therefore, $F[\overline{\bu}] \le F[\bu]$ for all $\bu \in \mathcal{W}^p_{\beta}(\Omega \cup \Gamma)$ and thus $\overline{\bu}$ is a minimizer of \eqref{Fu} with respect to the admissible function space $\mathcal{W}^p_{\beta}(\Omega \cup \Gamma)$.
\end{proof}

\subsection{Connection to classical Euler-Lagrange equations}\label{CEL}

Recall for a functional of the form
\begin{equation}\label{cEL}
\mathcal{F}[\bu]=\int_\Omega f(\bx,{\bu},\nabla {\bu}(\bx))\dd \bx,
\end{equation}
with $f:\re^n\times \re^N \times \re^{N\times n}$ for $\bu$ in some admissible class $\mc{A}$, and appropriate growth conditions on $f$, any minimizer of \eqref{cEL} satisfied the system of partial differential equations
\[
f_{u_i}(\bx,\bu(\bx),\nabla \bu(\bx))-\text{div}(f_{\bm{\xi_i}}(\bx,{\bu}(\bx),\nabla {\bu}(\bx)))=0,
\]
where $\bm\xi_i=(\xi_{ij})_{j=1}^N$ on $\Omega$.  If we allow a specific form of our nonlocal functional, one can see the similarities between the two conditions using nonlocal operators.  Let $g(\bx,\bu,\bm\eta): \re^n \times \re^N \times \re^{N \times N} \to \re$ satisfy one of the growth conditions above.  Let
\[
f(\bx,\bz,\bm{u},\bm{\xi})=g(\bx,\bm{u},\bm\xi \otimes \bm\mu(|\bz|)),
\]
 where $\bm{\mu}$ is an integrable kernel. Then, the weak form of the Euler-Lagrange equations for the functional
\[
\int_{\Omega \cup \Gamma}\int_{\Omega \cup \Gamma}g(\bx,\bm{u}(\bx),\widehat{\bu}(\bx,\by) \otimes\bm{\mu}(|\by-\bx|))\dd \bx\dd \by
\]
is given by
\begin{align}\label{CEL}
\int_{\Omega\cup\Gamma'} \bm\varphi(\bx) \int_{\Omega \cup \Gamma} g_{\bu}(\bx,\by)-[g_{\bm\eta}(\bx,\by)-g_{\bm\eta}(\by,\bx)]\bm\mu(|\by-\bx|)\dd \by\dd \bx=0,
\end{align}
for all appropriate $\bm{\varphi}$.  Using the definition for the nonlocal divergence, as stated in \cite{DGLz2013}, we can rewrite \eqref{CEL} as
\[
\int_{\Omega\cup \Gamma'} \bm\varphi(\bx)\int_{\Omega\cup\Gamma }g_{{\bu}}(\bx,\by)\dd \by\dd \bx-\int_{\Omega\cup\Gamma'} \bm\varphi(\bx) \mc{D}_{\bm\mu}[g_{\bm\eta}](\bx)\dd \bx=0.
\]
If $g$ has enough regularity, e.g. the integrand of $\eqref{CEL}$ is in the dual space of $\bm{\varphi}$, then for a.e. $\bx \in \Omega \cup \Gamma'$
\begin{align}
\begin{aligned}\label{CEQeq}
0&=\int_{\Omega \cup \Gamma} g_{\bu}(\bx,\by)-[g_{\bm\eta}(\bx,\by)-g_{\bm\eta}(\by,\bx)]\bm\mu(|\by-\bx|)\dd \by\\
&=\int_{\Omega \cup \Gamma} g_{\bu}(\bx,\by)\dd \by-\mc{D}_{\bm\mu}[g_{\bm\eta}](\bx).
\end{aligned}
\end{align}

This form is similar to the strong form of the Euler-Lagrange equations for \eqref{cEL} as shown above.  However, it is interesting to note a difference between the two forms. In the case of free boundary conditions ($|\Gamma'| \neq 0$), we have $\bu$ satisfying the Euler-Lagrange equations in a larger space $\Omega \cup \Gamma'$.  In the classical formulation, $\bu$ still satisfies the Euler-Lagrange equations on the interior, while Neumann (natural) boundary conditions could be collected from the integration by parts.

\section{Existence of minimizers}\label{EoM}

\begin{lem}\label{conlem} If $f$ is jointly convex in $\bu$ and $\bm{\xi}$ for almost every $\bx,\bz \in \Omega \cup \Gamma$, then $\mathcal{F}$ given by \eqref{Fu} is convex with respect to $\bu$.
\end{lem}

\begin{proof}
Let $\bu,\bv \in L^p(\Omega \cup \Gamma; \re^N)$ and $t \in [0,1]$.  Then, since $f$ is jointly convex in the last two variables, we have
\begin{multline*}
    f\Big(\bx,\by-\bx,t\bu(\bx)+(1-t)\bv(\bx),t\bu(\by)+(1-t)\bv(\by)
    -(t\bu(\bx)+(1-t)\bv(\bx))\Big) \\
    \le tf\Big(\bx,\by-\bx, \bu(\bx), \bu(\by)-\bu(\bx)\Big)+(1-t)f\Big(\bx,\by-\bx,\bv(\bx),\bv(\by)-\bv(\bx)\Big)
\end{multline*}
for almost every $\bx,\bz \in \Omega \cup \Gamma$.  Therefore, after integration we obtain
\begin{align*}
\mathcal{F}\Big(t\bu+(1-t)\bv\Big)
&\le \int  tf\Big(\bx,\by-\bx, \bu(\bx),\bu(\by)-\bu(\bx)\Big)\\
&\hspace*{.7in}+(1-t)f\Big(\bx,\by-\bx,\bv(\bx),\bv(\by)-\bv(\bx)\Big)\dd\by\dd \bx\\
&=t\mathcal{F}(\bu)+(1-t)\mathcal{F}(\bv),
\end{align*}
as desired.
\end{proof}

\begin{thm}\label{jc} Suppose that there is a $\delta>0$ such that
\begin{equation}\label{Eq:posDelta}
    \bigcup_{\bx\in\Omega}B_\delta(\bx)\subseteq\Omega\cup\Gamma.
\end{equation}
Let $f=f(\bx,\bz,\bu,\bm{\xi}):(\Omega \cup \Gamma) \times \widehat{\Omega \cup \Gamma} \times \re^N \times \re^N \to \re$ be given. Assume that $f$ is
\begin{itemize}
	\item continuous in $\bu$ and $\bm{\xi}$,
	\item measurable in $\bx$ and $\by$,
	\item jointly convex with respect to $\bu$ and $\bm{\xi}$ for almost every $(\bx,\bz) \in (\Omega \cup \Gamma)\times \widehat{\Omega \cup \Gamma}$,
	\item coercive:
\begin{equation}\label{coereq}
f(\bx,\by-\bx,\bu,\bm{\xi})\ge \alpha_1(\bx,\by-\bx)|\bm\xi|^p+\alpha_2(\bx,\by-\bx)|\bu|^q+\alpha_3(\bx,\by-\bx).
\end{equation}
\end{itemize}
For the coercivity condition, we require
\begin{itemize}
\item $\alpha_1: \Omega \cup \Gamma \times \widehat{\Omega \cup \Gamma} \to \re^N$ is a measurable function for which there exists a constant $C_0>0$ such that for all $\bx\in\Omega$ and $\by \in \Omega\cup\Gamma$
\[
\alpha_1(\bx,\by-\bx)\ge C_0\chi_{B_\delta(\bx)},
\]
with $\delta>0$ identified in \eqref{Eq:posDelta},
\item $\alpha_2 \in L^{p/(p-q)}[(\Omega \cup \Gamma)^2]$,
\item $\alpha_3 \in L^{1}[(\Omega \cup \Gamma)^2]$, and $1 \le q <p$.
\end{itemize}

Consider the functional \eqref{Fu}, and let $\bu_0\in W^p_{\alpha_1}(\Omega\cup\Gamma)$ such that $\mathcal{F}(\bu_0)<\infty$ be given. Set
\[
	\mathcal{A}:=\left\{\bv\in W^p_{\alpha_1}(\Omega\cup\Gamma;\re^N)
        : [\bv-\bu_0]\in W^p_{\alpha_1,0}(\Omega\cup\Gamma;\re^N)\right\}.
\]
Then there is some $\bu\in A$ such that
\begin{equation}\label{infeq}
\mathcal{F}(\bu)=\inf\{\mathcal{F}(\bv): \bv \in \mathcal{A}\}.
\end{equation}
\end{thm}

\begin{proof}

Define $m$ to be the infimum in \eqref{infeq}.  Then, since $\mathcal{F}(\bu_0)<\infty$ by assumption, $m<\infty$.  Also, by the coercivity condition \eqref{coereq} we have that $m>-\infty$ and so $m$ is finite.  Let $\{\bu_\nu\}$ be a minimizing sequence in $\mathcal{A}$.  Using \eqref{coereq}, for sufficiently large $\nu$ we have

\begin{align*}
    m+1
    &\ge
    \int_{(\Omega \cup \Gamma)^2} \Big[\alpha_1(\bx,\by-\bx)|\widehat{\bu}_\nu(\bx,\by)|^p
        -|\alpha_2(\bx,\by-\bx)||\bu_{\nu}(\bx)|^q\\
    &\hspace*{2in}\phantom{|\widehat{\bu}_\nu(\bx,\by)|^p}
        -|\alpha_3(\bx,\bx-\by)|\Big]\dd \bx\dd \by \\
    &=
    \|\widehat{\bu}_\nu\|_{L_{\alpha_1}^p[(\Omega \cup \Gamma)^2]}^p
    -|\Omega\cup\Gamma|^{\frac{q}{p}}
    \|\alpha_2\|_{L^\frac{p}{q-p}[(\Omega\cup\Gamma)]^2}\|\bu_\nu\|_{L^p(\Omega\cup\Gamma)}^q\\    
    &
    \hspace*{3.4in}-\|\alpha_3\|_{L^1[(\Omega\cup\Gamma)]^2}.
\end{align*}
Since $\bu_\nu\in\mathcal{A}$, we have $\bu_\nu(\bx)=\bu_0(\bx)$, for each $\bx\in\Gamma$. It follows that
\[
    \|\bu_\nu\|_{L^p(\Omega\cup\Gamma)}^q
    \le
    C_1\left(\|\bu_\nu\|_{L^p(\Omega)}^q+\|\bu_0\|_{L^p(\Gamma)}^q\right).
\]
Thus, with $C_2:=C_1|\Omega\cup\Gamma|^{\frac{q}{p}}\|\alpha_2\|_{L^\frac{p}{q-p}[(\Omega\cup\Gamma)]^2}$
and $C_3:=\|\alpha_3\|_{L^1[(\Omega\cup\Gamma)]^2}$,
\begin{equation}\label{Eq:Est41}
    \|\widehat{\bu}_\nu\|_{L^p_{\alpha_1}[(\Omega\cup\Gamma)^2]}^p
    -C_2\|\bu_\nu\|_{L^p(\Omega)}^q
    \le m+1+C_2\|\bu_0\|_{L^p(\Gamma)}^q+C_3.
\end{equation}
At this point, since $\alpha_1(\bx,\by-\bx)\ge C_0\chi_{B_\delta(\bx)}$ for each $\bx,\by\in\Omega\cup\Gamma$, we may apply the nonlocal Poincar\'e's inequality (see \cite[Lemma 3.5]{HR}) to find
\[
    \|\widehat{\bu}_{\nu}\|_{L^p_{\alpha_1}[(\Omega \cup \Gamma)^2]}^p \ge C_4\left(\|\bu_\nu\|_{L^p(\Omega)}^p-\|\bu_0\|_{L^p(\Gamma)}^p\right).
\]
Here $C_4>0$ depends on the structural assumptions in the problem. In particular, it depends on $C_0$ and $\delta$. Using this in \eqref{Eq:Est41} yields
\begin{equation}\label{Eq:Est42}
    C_4\|\bu_\nu\|_{L^p(\Omega)}^p-C_2\|\bu_\nu\|_{L^p(\Omega)}^q
    \le
    m+1+C_2\|\bu_0\|_{L^p(\Gamma)}^q+C_4\|\bu_0\|_{L^p(\Gamma)}^p+C_3.
\end{equation}
Since $q<p$, Young's inequality implies that for each $\vep>0$, there is a $C<\infty$ such that
\begin{equation}\label{Eq:YoungsBnd}
    \|\bu_\nu\|_{L^p(\Omega)}^q
    \le
    \varepsilon\|\bu_\nu\|_{L^p(\Omega)}^p+C.
\end{equation}
With $\varepsilon=\ds\frac{C_4}{2C_2}$, it follows from \eqref{Eq:Est42} that
\[
    \frac{1}{2}C_4\|\bu_\nu\|_{L^p(\Omega)}^p
    \le
    C_2\|\bu_0\|_{L^p(\Gamma)}^q+C_4\|\bu_0\|_{L^p(\Gamma)}^p+C_5,
\]
for some $C_5<\infty$. Hence
\begin{equation}\label{Eq:Est43}
    \frac{1}{2}C_4\|\bu_\nu\|_{L^p(\Omega\cup\Gamma)}^p
    \le \frac{3}{2}C_4\|\bu_0\|_{L^p(\Gamma)}^p+C_5.
\end{equation}
Then both $C_4$ and the upper bound is independent of $\nu$, and recall that $C_4>0$. Thus we can extract a (unrelabeled) subsequence $\{\bu_\nu\}$ such that $\bu_\nu \to \overline{\bu}$ weakly in $L^p(\Omega\cup\Gamma)$ for some $\overline{\bu}\in L^p(\Omega\cup\Gamma)$.

We next verify that $\overline{\bu}\in\mathcal{A}$. Clearly $\overline{\bu}=\bu_0$ on $\Gamma$, so we just need to show that $\overline{\bu}\in W^p_{\alpha_1}(\Omega\cup\Gamma)$. By Mazur's Lemma, there is a sequence of convex combinations of elements of $\{\bu_\nu\}$ that converge strongly to $\overline{\bu}$ in $L^p(\Omega\cup\Gamma)$. Thus we may produce $\{\overline{\bu}_\kappa\}_{\kappa=1}^\infty$ in $L^p(\Omega\cup\Gamma)$ such that
\[
    \overline{\bu}_\kappa=\sum_{\nu=\kappa}^{N_{\kappa}}\lambda_\nu^{({\kappa})}\bu_\nu
    \quad
    \text{ and }
    \quad
    \text{ $\overline{\bu}_\kappa\to \overline{\bu}$ strongly in $L^p(\Omega\cup\Gamma)$}
\]
where for all $\kappa\in\mathbb{N}$,
\[
    \{\lambda_\nu^{({\kappa})}\}_{\nu,\kappa=1}^\infty\in[0,1]
    \quad\text{ and }
    \sum_{\nu=\kappa}^{N_\kappa}\lambda_\nu^{({\kappa})}=1.
\]
Thus, after passing to another (unrelabeled) subsequence, $\overline{\bu}_\nu\to\overline{\bu}$ pointwise a.e. in $\Omega\cup\Gamma$. It follows that
\begin{multline*}
    \alpha_1(\bx,\by-\bx)\widehat{\overline{\bu}}_\kappa(\bx,\by)
    =\alpha_1(\bx,\by-\bx)\left[\overline{\bu}_\kappa(\by)
        -\overline{\bu}_\kappa(\bx)\right]\\
    \to
    \alpha_1(\bx,\by-\bx)\widehat{\overline{\bu}}(\bx,\by),
    \quad
    \text{a.e. }\bx,\by\in\Omega\cup\Gamma.
\end{multline*}
By Fatou's lemma
\[
    \|\widehat{\overline{\bu}}\|_{L^p_{\alpha_1}[(\Omega\cup\Gamma)^2]}^p
    \le
    \liminf_{\kappa\to\infty}
    \|\widehat{\overline{\bu}}_\kappa\|_{L^p_{\alpha_1}[(\Omega\cup\Gamma)^2]}^p.
\]
Since each $\overline{\bu}_\kappa$ is a convex combination of elements of $\{\bu_\nu\}$, we see that $\overline{\bu}_\kappa=\bu_0$ on $\Gamma$. Since $t\mapsto|t|^p$ is a convex function, we also find that
\begin{align*}
    \|\widehat{\overline{\bu}}_\kappa\|_{L_{\alpha_1}^p[(\Omega\cup\Gamma)^2]}^p
    &=
    \int_{(\Omega\cup\Gamma)^2}\alpha_1(\bx,\by-\bx)
        \left|\sum_{\nu={\kappa}}^{N_{\kappa}}\lambda_\nu^{({\kappa})}
            \widehat{\bu}_\nu\right|^p
        \dd\bx\dd\by\\
    &\le
    \sum_{\nu={\kappa}}^{N_{\kappa}}\lambda_\nu^{({\kappa})}
        \|\widehat{\bu}_\nu\|_{L_{\alpha_1}^p[(\Omega\cup\Gamma)^2]}^p.
\end{align*}
Next, from the bound in \eqref{Eq:Est41}, we conclude that
\[
    \|\widehat{\overline{\bu}}_\kappa\|_{L_{\alpha_1}^p[(\Omega\cup\Gamma)^2]}^p
    \le
    m+1+C_2\left(\|\bu_0\|_{L^p(\Gamma)}^q+\sum_{\nu=\kappa}^{N_{\kappa}}
        \lambda_\nu^{({\kappa})}
        \|\bu_\nu\|_{L^p(\Omega)}^q\right)
    +C_3.
\]
An application of Young's inequality, as in \eqref{Eq:YoungsBnd}, and the bound in \eqref{Eq:Est43} establishes an upper bound for $\|\widehat{\overline{\bu}}_\kappa\|_{L^p_{\alpha_1}[(\Omega\cup\Gamma)^2]}^p$ that is independent of $\kappa\in\mathbb{N}$. Hence $\overline{\bu}\in W^p_{\alpha_1}(\Omega\cup\Gamma)$. As already indicated, we also have that $\overline{\bu}=\bu_0$ on $\Gamma$. Thus $\overline{\bu}\in\mathcal{A}$, as claimed.

Finally, we verify that $\overline{\bu}$ minimizes $\mathcal{F}$ over $\mathcal{A}$. For each $\kappa\in\mathbb{N}$, denote by $\widetilde{\bu}_\kappa$ that element of $\{\bu_\nu\}_{\nu=\kappa}^{N_\kappa}$ such that $\mathcal{F}(\widetilde{\bu}_\kappa)=\max_{\kappa\le\nu \le N_\kappa}\mathcal{F}(\bu_\nu)$. By Lemma \ref{conlem} and the convexity assumption for $f$,
\begin{equation}\label{Eq:FukappaBnd}
    \mathcal{F}(\overline{\bu}_\kappa)
    =
    \mathcal{F}\left(\sum_{\nu=\kappa}^{N_\kappa}
        \lambda_\nu^{(\kappa)}\bu_\nu\right)
    \le
    \sum_{\nu=\kappa}^{N_\kappa}\lambda_\nu^{(\kappa)}\mathcal{F}(\bu_\nu)
    \le \left(\sum_{\nu=\kappa}^{N_\kappa}\lambda_\nu^{(\kappa)}\right)
        \mathcal{F}(\widetilde{\bu}_\kappa)=\mathcal{F}(\widetilde{\bu}_\kappa).
\end{equation}
Moreover, since $\{\overline{\bu}_\kappa\}$ was identified so that $\overline{\bu}_\kappa\to\overline{\bu}$ a.e. in $\Omega\cup\Gamma$, the continuity and measurability assumptions on $f$, imply
\[
    f(\bx,\by-\bx,\overline{\bu}_\kappa(\bx),\widehat{\overline{\bu}}_\kappa(\bx,\by))
    \to
    f(\bx,\by-\bx,\overline{\bu}(\bx),\widehat{\overline{\bu}}(\bx,\by))
    \quad\text{ a.e. }\bx,\by\in\Omega\cup\Gamma.
\]
By Fatou's Lemma
\[
    \mathcal{F}(\overline{\bu})
    \le
    \liminf_{\kappa\to\infty}\mathcal{F}(\overline{\bu}_\kappa)
    \le
    \liminf_{\kappa\to\infty}\mathcal{F}(\widetilde{\bu}_\kappa).
\]
Since $\widetilde{\bu}_\kappa$ is a subsequence of our original minimizing sequence, we deduce that
\[
    \mathcal{F}(\overline{\bu})
    =\inf\left\{\mathcal{F}(\bv):\bv\in\mathcal{A}\right\}.
\]
Thus, $\overline{\bu}$ is a minimizer of \eqref{Fu}.
\end{proof}
The next couple of remarks indicate some generalizations of Theorem~\ref{jc}. The key components of the argument are uniform bounds for the $L^p$ norm of the minimizing sequence, as in \eqref{Eq:Est43}, and the convexity of $\mathcal{A}\subseteq W^p_{\alpha_1}$, so that Mazur's lemma can be used.
\begin{rem}
One extension of Theorem~\ref{jc} is to allow admissible functions with only partially prescribed values on $\Gamma$. More precisely, there is a $\Gamma'\subseteq\Gamma$ such the admissible class is
\[
    \mathcal{A}:=\left\{\bv\in W^p_{\alpha_1}(\Omega\cup\Gamma;\re^N)
        :[\bv-\bu_0]\in W^p_{\alpha_1,\Gamma'}(\Omega\cup\Gamma;\re^N)\right\}.
\]
Thus $\bv\in\mathcal{A}$ have unspecified values on $\Gamma\setminus\Gamma'$. The argument for the bound in \eqref{Eq:Est43} is the same, but an additional coercivity term is required to compensate for the uncontrolled boundary component. For example, we could modify \eqref{coereq} to
\[
    f(\bx,\by-\bx,\bu,\bm{\xi})\ge \alpha_1(\bx,\by-\bx)|\bm\xi|^p
    +\alpha_2'(\bx,\by-\bx)|\bu|^p,
    +\alpha_3(\bx,\by-\bx),
\]
with $\alpha_2'\in L^\infty[(\Omega\cup\Gamma)^2]$ satisfying $\alpha_2'\ge M$ uniformly on $(\Omega\cup\Gamma)^2$, for some sufficiently large $M<\infty$. The lower bound on $M$ depends on the other structural assumptions for $f$ and, in particular, the constant that appears in the nonlocal Poincar\'{e} inequality.
\end{rem}
\begin{rem}
Another extension is to impose constraints on the admissible function that preserve the convexity of the admissible class. For example, we could seek a minimizer for $\mathcal{F}$ over
\[
    \mathcal{A}:=\left\{\bv\in W^p_{\alpha_1}(\Omega\cup\Gamma;\re^N)
        :   \|\bv\|_{L^p(\Gamma;\re^N)}\le M\right\}.
\]
Thus the admissible functions have unprescribed boundary values but there is a control on the $L^p$ size of the boundary portion of the function. For this type of problem, no modification of the coercivity condition on $f$ is needed. The boundary terms in \eqref{Eq:Est43} are bounded by assumption.
\end{rem}

\section{Regularity of minimizers}\label{RoM}

We have existence and uniqueness of minimizers for the energy functional \eqref{NLF} as it follows from Theorem \ref{jc}, hence, we have a solution for the nonlocal boundary value problem. 

In this section we show that under mild assumptions for the kernel the solutions to semilinear problems have $W^{1,1}$ regularity, then $C^{\infty}$.

\begin{thm}[$W^{1,1}$ regularity of solutions]\label{RoS}
Let $\mu\in W^{1,1}(\widehat{\Omega \cup \Gamma};\re)$ (with the notation of \eqref{hatset})  
and $\|\mu\|_1=1$. Let $\bm f_0: \re^n \times \re^N \to \re^N$ 
for which we assume that there exists a $\bm g \in W^{1,\infty}(\re^n \times \re^N ;\re^N)$ such that for each $\bx \in \re^n$ and $\bu \in \re^N$ we have 
\begin{equation}\label{inv}
\bm g(\bx, \bu+\ds\frac{1}{2}\bm f_0(\bx,\bu))=\bu.
\end{equation}
Let $\bu \in L^1(\Omega\cup \Gamma;\re^N)$ be a solution of
\[
\mathcal{L}_{\mu}[\bu](\bx)=\bm f_0(\bx,\bu),
\]
for every $\bx \in \Omega$, with
\[
\mathcal{L}_{\mu}[\bu](\bx)=2\int_{\Omega \cup \Gamma} \left(\bu(\by)-\bu(\bx)\right)\mu(\by-\bx) \dd \by.
\]
 Then under the above assumptions
\[
\bu \in W^{1,1}(\Omega;\re^N).
\]
\end{thm}

\begin{proof}

Since $\bu$ is a solution to the given integral equation, the definition of the nonlocal Laplacian (upon trivially extending $u$ and $\mu$ by zero outside $\Omega\cup\Gamma$) gives
\begin{equation}\label{nonlf}
(\bu*\mu)(\bx)-\bu(\bx)=\frac{1}{2}\bm f_0(\bx,\bu(\bx)), \ \bx \in \Omega,
\end{equation}
where $\bu *\mu =(u_1*\mu, u_2*\mu, ..., u_N*\mu).$ Then, by assumption \eqref{inv} there is a $\bm g\in W^{1,\infty}(\re^n \times \re^N; \re^N)$ such that for each $\bx \in \Omega$ 
\begin{equation}\label{bseq}
\bm g(\bx, (\bu*\bm{\mu})(\bx))=\bu(\bx).
\end{equation}
Since $\mu\in W^{1,1}(\widehat{\Omega \cup \Gamma};\re)$ and $\bu \in L^1(\Omega\cup \Gamma;\re^N)$ we have that $\bu *\mu\in W^{1,1}(\Omega;\re)$. By the regularity of $\bm g$, we have $\bx \mapsto \bm g(\bx,(\bu*\bm{\mu})(\bx)) \in W^{1,1}(\Omega; \re^N)$ by the chain rule for Sobolev functions (see \cite{Mizel}, and \cite{leoni} for a more general result). Thus
\[
\bu \in W^{1,1}(\Omega; \re^N).
\] 
\end{proof}

Assuming additional regularity for the function $\bm g$ (which is basically the inverse of $\bm f$ with respect to the $\bu$ variable) and for the kernel $\mu$ yields additional regularity for the solution $\bu$. The result below shows that one could eventually obtain $C^{\infty}$ regularity; it is easily seen that in fact $W^{k,1}$ regularity for $\bm g$ and $\mu$ are matched by $W^{k,1}$ regularity for $\bu$. 
\begin{thm}[$C^{\infty}$ regularity of solutions]\label{RoS1}
Under the assumptions of Theorem \ref{RoS}, additionally assume that $\bm g \in C^{\infty}(\re^n \times \re^N ;\re^N)$ and \\
$\mu\in C^{\infty}(\widehat{\Omega\cup\Gamma};\re)$. Then
\[
\bu \in C^{\infty}(\Omega; \re^N).
\] 
\end{thm}

\begin{proof}
By Theorem \ref{RoS} we have $\bu \in W^{1,1}(\Omega; \re^N)\cap L^1(\Omega\cup \Gamma)$. Using again the regularity of $\mu$ it follows that $\bu*\bm{\mu} \in W^{2,1}(\Omega;\re^N)$ and so by \eqref{bseq} and by using the chain rule again, we have that $\bu \in W^{2,1}(\Omega;\re^N)$. Iterating this argument $k$ times for $k\in \na$, we obtain $\bm u \in W^{k,1}(\Omega; \re^N)$.  
The Sobolev embedding theorem will then yield that $\bu \in C^\infty(\Omega;\re^N)$.

\end{proof}

\subsection{Example on existence, uniqueness, and regularity of solutions.} For simplicity consider a scalar case with $n=N=1$; the arguments can be easily extended to the vectorial setting. Set the domain $\Omega=(-1,1)$ with $\Gamma=\re \setminus (-1,1)$. Let 
\begin{equation}\label{gauss}
\mu(x)=\ds\frac{1}{\sqrt{\pi}}e^{-x^2}
\end{equation} for which $\|\mu\|_{L^1(\re)}=1$.  Let $f_0:(x,u)\in\re\times \re \to \re$
\[
f_0(x,u)=2\frac{\arctan{u}+1}{x^2+1}.
\] 
The equation
\begin{equation}\label{eq4}
\mc{L}_{\mu}[u](x)=2\frac{\arctan{u}(x)+1}{x^2+1}, \quad x\in (-1,1)
\end{equation}
for functions $u$ with $u=0$ on $\Gamma=\re \setminus (-1,1)$ is the Euler-Lagrange equation associated with the energy functional given by 
\begin{align*}
\mathcal{F}[u]&=\Int_{\re}\Int_{\re} [u(y)-u(x)]^2 \mu(y-x) \\
&\hspace*{1.5in}+\frac{2u(x)\arctan u(x)-\ln[1+u^2(x)]+u(x)}{x^2+1} \, dy dx.
\end{align*}
Indeed, since the regularity assumptions and the growth assumption (GI) of Theorem \ref{NELE} are satisfied by the integrand of the energy functional, we have that the minimizers satisfy the Euler-Lagrange equation \eqref{eq4}. To show existence and uniqueness of minimizers observe that the functional form of the integrand given by
\[
 f(x, z, u,\xi)= \xi^2 \mu(z) +\frac{2u \arctan u-\ln(1+u^2)+u}{x^2+1}
\]
clearly satisfies the continuity and measurability assumptions of Theorem \ref{EoM}. Moreover, it is jointly convex in $u$ and $\xi$, as it is the sum of a convex function in $u$ and a convex function in $\xi$ (the convexity in $u$ can be easily verified through the positivity of the second order derivative; indeed, $\ds\frac{1}{(1+u^2)(x^2+1)} \geq 0$ for all $x$ and $u$). For coercivity (as required by Theorem \ref{EoM}) note first that
\[
\frac{2u \arctan u-\ln(1+u^2)+u}{x^2+1} \geq 0
\]
for all $x,u \in \re$. Thus we have
\[
|f(x, z, u,\xi)| \geq c_1 |\xi|^2
\]  
for all $|z|\leq 1/2$ (take $\delta=1/2$ in assumption \eqref{Eq:posDelta}). To study the regularity of solutions we observe that $\mu\in C^{\infty}(\re)$. By \eqref{nonlf} we obtain
\begin{equation}\label{eq5}
(u *\mu)(x) =\frac{\arctan{u}(x)+1}{x^2+1} +u(x).
\end{equation}
Since the function 
\[
(x,u)\mapsto \ds\frac{\arctan{u}+1}{x^2+1} +u
\]
is $C^{\infty}$ and its derivative with respect to $u$ is uniformly positive, by the implicit function theorem, there is a $C^\infty$ function $g:\re\times\re\to\re$ such that
\[
	\frac{\arctan{g(x,u)}+1}{x^2+1} +g(x,u)=u
\] 
for all $x,u\in \re$. Thus, from \eqref{eq5} we have
\[
u(x)=g(x,(u*\mu)(x)),
\]
so Theorem \ref{RoS1} implies $u\in C^{\infty}(\re)$.

The regularity theorem \ref{RoS} has an interesting consequence that regards {\it ill-posedness} of nonlocal problems. We present this aspect through an example below.

\subsection{Example on ill-posedness.}
 Let $\mu \in L^{\infty}(\Omega; \re)$ (for example, the Gaussian \eqref{gauss}). Assume that
\[
\bm f_0(\bx,\bu)=\bm h(\bx)-\bu,
\]
for some $\bm h \in L^1(\Omega; \re^N)\setminus  L^{\infty}(\Omega; \re^N).$ 
Then the Euler-Lagrange equation
\begin{equation}\label{illposed}
\mc{L}[\bu](\bx)+\bm f_0(\bx,\bu)=0
\end{equation}
becomes
\[
(\bu\ast \bm{\mu})(\bx) -\bu(x)=\bm h(\bx)-\bu(\bx),
\]
hence $(\bu\ast \bm{\mu})(\bx) =\bm h(\bx)$. As a consequence of Young's inequality for convolutions (which is the idea behind the higher integrability result of \cite[Theorem 3.1]{FR}) we have that
\[
\bu \ast \bm{\mu} \in L^{\infty}(\Omega; \re^N).
\]
 Therefore, we must have that $\bm h\in L^{\infty}(\Omega; \re^N)$, which is precluded by the choice of $\bm h$. Thus, the problem \eqref{illposed} does not admit a solution, which shows that even {\it linear} problems may be ill-posed in the nonlocal setting, when the source term $\bm f_0$ is not sufficiently regular to accommodate the gain in regularity given by $\bm\mu$.

\section{Examples}\label{Ex}

In the proceeding examples, assume $\bm\mu: \widehat{\Omega \cup \Gamma} \to \re^N$ is positive and symmetric. Also in each example $f:\Omega \cup \Gamma \times \widehat{\Omega \cup \Gamma} \times \re^N \times \re^N \to \re$ will be the integrand of the nonlocal functional \eqref{NLF}.  For each functional, we will require the boundary condition $\bu_0$ on $\Gamma \setminus \Gamma'$ where $\Gamma' \subseteq \Gamma$ for some $\bu_0 \in L^p(\Gamma \setminus \Gamma';\re^N)$.

\subsection{Quasilinear nonlocal example.}

Let $f(x,z,u,\xi)=\dfrac{|\bm\xi\otimes \bm\mu(\bz)|^p}{p}+G(\bx,\bu)$ where $\bm\mu \in L^p(\re^N)$ and $G: \Omega \cup \Gamma \times \re^N \to \re$ is convex for each $\bx$ with
\[
G(\bx,\bu) \ge a_1(\bx)|\bu|^q+a_2(\bx), \quad  \partial_u G(\bx,\bu) \le b_1(\bx)+b_2(\bx)|\bu|^{q}
\]
where $a_1\in L^{p/(p-q)}(\Omega \cup\Gamma), \ a_2 \in L^{1}(\Omega \cup\Gamma), b_1 \in L^{p/(p-1)}(\Omega \cup \Gamma), b_2 \in L^\infty(\Omega \cup\Gamma)$ for $1\le q < p$.  With these assumptions $f$ then satisfies the growth condition (GII) and Theorem \ref{jc}.  Thus, there is a minimizer $\bu \in \mc{W}^p_{\bm\mu}(\Omega \cup \Gamma)$ with $\bu=\bu_0$ on $\Gamma$ for the corresponding nonlocal functional, and it satisfies the nonlocal Euler-Lagrange equations. As in \eqref{CEL}
\[
 g(\bx,\bu,\bm\eta)=\dfrac{|\bm\eta|^p}{p}+G(\bx,\bu)
 \]
  and note $\partial_{\bm\eta}g(\bx,\bu,\bm\eta)=\bm\eta|\bm\eta|^{p-2}$. Then, using the fact $|\bm\xi \otimes \bm\mu|=|\bm\xi||\bm\mu|$ and \eqref{CEQeq}, we have
\begin{align*}
\int_{\Omega \cup \Gamma} g(\bx,\bu)\dd \by-\mc{D}\left[\widehat{\bu}(\bx,\by) \otimes \bm\mu(\by-\bx)|\widehat{\bu}(\bx,\by) \otimes \bm\mu(\by-\bx)|^{p-2}\right](\bx)=0
\end{align*}
for $\bx \in \Omega$.
Then, letting $g(\bx,\bu)=|\Omega \cup \Gamma| \partial_{\bu} G(\bx,\bu)$ and using $\eqref{G}$ and $\eqref{pL}$ we can rewrite the above as
\[
\mathcal{L}_{\bm\mu}^p[\bu](\bx)=g(\bx,\bu), \quad  a.e. \ \bx \in \Omega \cup \Gamma'.
\]
Thus, the minimizer of the functional is a solution to the $p$-Laplacian problem.

If we take $p=1$ in the above formulation and assume some more conditions on $G$ and $\bm\mu$ (so that $g$ and $\bm\mu$ satisfy the conditions of theorem \ref{RoS}) then we can achieve additional regularity on the minimizer $\bu$.





\subsection{Another quasilinear example with nonlocal $p$-Laplacian.}

Consider the functional
\[
f(\bx,\bz,\bu,\bm\xi)=|(\bu+\bm\xi)\otimes \bm\mu|^q+|\bm\xi\otimes \bm\mu|^p
\]
with $1\le q<p$ and $|\bm\mu|^q$ satisfying the same conditions as $\alpha_1$ and $\alpha_2$ in Theorem \ref{jc}.
Then note that $f$ is convex with respect to $(\bu,\bm\xi)$ for all $\bx$ and $\bz$, with
\begin{align*}
f(\bx,\bz,\bu,\bm\xi) &\ge |\bm\mu(\bz)|^q(2^{q-1}|\bm\xi|^q+|\bm\xi|^p|\bm\mu(\bz)|^p)-|\bm\mu(\bz)|^q2^{q-1}|\bu|^q
\\
&\ge |\bm\xi|^p|\bm\mu(\bz)|^p-|\bm\mu(\bz)|^q|\bu|^q.
\end{align*}
Also the corresponding nonlocal functional is finite for $\bu \in \mc{W}^{p}_{\bm\mu}(\Omega \cup \Gamma)$, so by Theorem \ref{jc}, there is a minimizer $\bu \in \mathcal{W}^p_{\bm\mu}(\Omega \cup \Gamma)$ with $\bu=\bu_0$ on $\Gamma$. Now,
\begin{align*}
\partial_{\bu} f(\bx,\by-\bx,\bu(\bx),\bu(\by)-\bu(\bx))&=q\bu(\by)\bm|\bu(\by)|^{q-2}|\bm\mu(\by-\bx)|^q;\\
\partial_{\bm\xi} f(\bx,\by-\bx,\bu(\bx),\bu(\by)-\bu(\bx))&=q\bu(\by)\bm|\bu(\by)|^{q-2}|\bm\mu(\by-\bx)|^q\\
&+p\widehat{\bu}(\by)\bm|\widehat{\bu}(\by)|^{p-2}|\bm\mu(\by-\bx)|^p.
\end{align*}

Also the corresponding nonlocal functional is finite for $\bu \in \mc{W}^{p}_{\bm\mu}(\Omega \cup \Gamma;\re^N)$, so by Theorem \ref{jc}, there is a minimizer $\bu \in \mathcal{W}^p_{\bm\mu}(\Omega \cup \Gamma)$ with $\bu=\bu_0$ on $\Gamma$. Now, for $\bu,\bm\xi\in \re^N$,
\begin{align}\label{parf}
\begin{aligned}
\partial_{\bu} f(\bx,\by-\bx,\bu,\bm\xi)&=q\bu\bm|\bu|^{q-2}|\bm\mu(\by-\bx)|^q;\\
\partial_{\bm\xi} f(\bx,\by-\bx,\bu,\bm\xi)&=q\bu\bm|\bu|^{q-2}|\bm\mu(\by-\bx)|^q+p\bm\xi\bm|\bm\xi|^{p-2}|\bm\mu(\by-\bx)|^p.
\end{aligned}
\end{align}
Note that for $q<p$, $|\bm\gamma|^{q-1} \le 1+|\bm\gamma|^{p-1}$ for $\bm\gamma \in \re^N$. Hence by \eqref{parf}
\begin{align*}
|\partial_{\bu} f(\bx,\by-\bx,\bu,\bm\xi)|&\le q\bm|\bu|^{q-1}|\bm\mu(\by-\bx)|^q \\
&\le q(1+|\bu|^{p-1})|\bm\mu(\by-\bx)|^q\\
&\le q|\bm\mu(\by-\bx)|^q+q|\bm\mu(\by-\bx)|^q(|\bu|^{p-1}+|\bm\xi|^{p-1});\\
|\partial_{\bm\xi} f(\bx,\by-\bx,\bu,\bm\xi)|&\le q\bm|\bu|^{q-1}|\bm\mu(\by-\bx)|^q+p\bm|\bm\xi|^{p-1}|\bm\mu(\by-\bx)|^p\\
&\le p\bm|\bu|^{q-1}(1+|\bm\mu(\by-\bx)|^p)+p\bm|\bm\xi|^{p-1}(1+|\bm\mu(\by-\bx)|^p)\\
&\le p(1+|\bm\mu(\by-\bx)|^p)\\
&\quad +p(1+|\bm\mu(\by-\bx)|^p)(|\bu|^{p-1}+|\bm\xi|^{p-1})
\end{align*}

Hence $f$ clearly satisfies the growth condition (GII) and so \eqref{ELe2} holds for $\bm{\varphi} \in \mathcal{W}^{p}_{\bm\mu,0}(\Omega \cup \Gamma)$.  Then, using the strong form of the Euler-Lagrange equations we have for a.e. $\bx \in \Omega \cup \Gamma$,
\begin{align*}
q\bu(\bx)\bm|\bu(\bx)|^{q-2}\|\bm\mu^q\|_{L^1}-2\int_{\Omega \cup \Gamma}\widehat{\bu}(\by)\bm|\widehat{\bu}(\by)|^{q-2}|\bm\mu(\by-\bx)|^q\dd y=0.
\end{align*}
Using \eqref{pL} and letting $M_{p,q}=\ds\frac{p}{q}\|\bm\mu^q\|_{L^1}$, we can rewrite this as
\begin{equation}\label{qLe}
\mathcal{L}_{\bm\mu}^q[\bu](\bx)=M_{p,q}\bu(\bx)|\bu(\bx)|^{p-2} \ a.e. \ \bx \in \Omega \cup \Gamma .
\end{equation}

\subsection{Semilinear case.}

For this case we will consider $N=1$ and the integrand of \eqref{NLF} be given by
\begin{equation}\label{QI}
f(\bx,\bz,u,\xi)=G(\bx,u)+2u\xi\mu(\bz)+[\xi\mu(\bz)]^2,
\end{equation}
where $\mu$ satisfies the same conditions as $\alpha_1$ in Theorem \ref{jc}, and $\mu^2(\bx)>\mu(\bx)$, for every $\bx\in \re^n$. Denote $\partial_u G(\bx,u)=g(\bx,u)$, and let $G$ be such that $f$ given by \eqref{QI} is convex with respect to $(u,\xi)$. Also, suppose that $|g(\bx,u)| \le C|u|^p$ and $G(\bx,u) \ge C|u|^p$ for some $p >1$.  Then the corresponding nonlocal functional satisfies Theorem \ref{jc}, and hence there is a minimizer $u \in \mathcal{W}^p_{\mu}(\Omega \cup\Gamma)$ with $u=u_0$ on $\Gamma$.  Also, $f$ satisfies growth condition (GII), so we can write down the Euler-Lagrange equation for this functional.

\begin{prop}\label{ELQ}  The strong form of the Euler-Lagrange equations associated with the functional \eqref{QI} can be written as
\begin{equation}\label{elq1}
 \quad (u*\gamma)(\bx)-u(\bx)=Mg(\bx,u(\bx)) \quad a.e. \ \bx \in \Omega \cup \Gamma',
\end{equation}
where
\[
M=\|\mu\|_{L^2}-\|\mu\|_{L^1} >0, \  C=\frac{|\Omega \cup \Gamma|}{2M}, \ \gamma(\bz):=\frac{\mu(\bz)^2-\mu(\bz)}{M}>0, \ \bz \in \widehat{\Omega \cup \Gamma}
\]
\end{prop}
\begin{proof} First, note
\begin{align*}
f_{u}(\bx,\by)&=g(\bx,u(\bx))+2 \mathcal{G}_{\mu}[u](\bx,\by);\\
f_{\xi}(\bx,\by)&=2u(\bx)\mu(\bz)+2 \mathcal{G}_{\mu^2}[u](\bx,\by).
\end{align*}
The strong form of the Euler-Lagrange equations \eqref{sEL} written for the functional \eqref{QI}  give that for a.e. $\bx \in \Omega \cup \Gamma$ we have
\begin{align*}
\int_{\Omega\cup\Gamma} &g(\bx,u(\bx))+2\mathcal{G}_{\mu}[u](\bx,\by)-\left[(2u(\bx)\mu(\by-\bx) \right.\\
&\left. +2\mathcal{G}_{\mu^2}[u](\bx,\by))-(2u(\by)\mu(\by-\bx)+2\mathcal{G}_{\mu^2}[u](\by,\bx))\right]\dd \by=0.
\end{align*}
Combining terms and using the definition of the nonlocal gradient this is
\begin{align*}
|\Omega \cup \Gamma|g(\bx,u(\bx))&+2\int_{\Omega \cup \Gamma}\left(\mathcal{G}_{\mu^2}[u](\by,\bx)-\mathcal{G}_{\mu^2}[u](\bx,\by)\right)\dd \by\\
&+4\int_{\Omega\cup\Gamma}\mathcal{G}_{\mu}[u](\bx,\by)\dd \by=0.
\end{align*}
Now, using the definition of the nonlocal Laplacian we obtain
\[
|\Omega \cup \Gamma|g(\bx,u(\bx))-2\mathcal{L}_{\mu^2}[u](\bx)+4\int_{\Omega\cup\Gamma}\mathcal{G}_{\mu}[u](\bx,\by)\dd \by=0.
\]
We can then write the above integrals using convolutions and combine terms,
\[
|\Omega \cup \Gamma|g(\bx,u(\bx))+4(\|\mu^2\|_1-\|\mu\|_1)u(\bx)-4(u*(\mu^2-\mu))(\bx)=0.
\]
Finally, using $\gamma$, $C$, and $M$ defined above, we obtain
\begin{align*}
\mathcal{L}_{\gamma}[u](\bx)=Cg(\bx,u(\bx)), \quad a.e. \ \bx \in \Omega \cup \Gamma.
\end{align*}

\end{proof}


\section{Conclusions} 
To conclude, we have identified here two growth conditions for the Euler-Lagrange equations to hold in the nonlocal setting. We have shown existence of minimizers of energy functionals with {\it no upper growth assumptions}, but under convexity and coercivity hypothesis. Other existence results in this area limit the upper bound on the growth to a power-type growth, and this upper bound must match corresponding power-type lower bound. The flexibility for our integrand allows to have, for example, exponential, variable, or anisotropic growth. 

We have also shown regularity of solutions to semilinear problems with suitably chosen nonlinearities, under mild assumptions on the kernel. This case is of interest to us, as for the nonlinear setting, convergence of nonlocal solutions to classical counterparts when the horizon of interaction shrinks to zero, remains an open problem.  The argument of Theorem \ref{RoS} based on properties of convolutions will be utilized in future work, as it provides a roadmap to regularity results under very general assumptions.

Also, the paper includes applications of the results to several settings, including different nonlinearities (of quasilinear and semilinear structure), in both scalar, and vectorial settings.
Connections to classical setting will be studied in future work from a theoretical point of view, as well as from a numerical stance, by exhibiting convergence rates of nonlocal solutions to classical counterparts.


\begin{thebibliography}{10}

\bibitem{Rossi}
F.~Andreu-Vaillo, J.~M. Maz{\'o}n, J.~D. Rossi, and J.~J. Toledo-Melero.
\newblock {\em Nonlocal diffusion problems}, volume 165 of {\em Mathematical
  Surveys and Monographs}.
\newblock American Mathematical Society, Providence, RI; Real Sociedad
  Matem\'atica Espa\~nola, Madrid, 2010.

\bibitem{BellidoMora}
J.~C. Bellido and C.~Mora-Corral.
\newblock Existence for nonlocal variational problems in peridynamics.
\newblock {\em SIAM J. Math. Anal.}, 46(1):890--916, 2014.

\bibitem{BMP1}
J.~C. Bellido, C.~Mora-Corral, and P.~Pedregal.
\newblock Addendum to the paper: {H}yperelasticity as a {$\Gamma$}-limit of
  peridynamics when the horizon goes to zero [{MR}3396427].
\newblock {\em Calc. Var. Partial Differential Equations}, 54(2):1671, 2015.

\bibitem{BMP}
J.~C. Bellido, C.~Mora-Corral, and P.~Pedregal.
\newblock Hyperelasticity as a {$\Gamma$}-limit of peridynamics when the
  horizon goes to zero.
\newblock {\em Calc. Var. Partial Differential Equations}, 54(2):1643--1670,
  2015.

\bibitem{BP}
J.~Bevan and P.~Pedregal.
\newblock A necessary and sufficient condition for the weak lower
  semicontinuity of one-dimensional non-local variational integrals.
\newblock {\em Proceedings of the Royal Society of Edinburgh: Section A
  Mathematics}, 136(04):701--708, 2006.

\bibitem{BV}
C.~Bucur and E.~Valdinoci.
\newblock {\em Nonlocal diffusion and applications}, volume~20 of {\em Lecture
  Notes of the Unione Matematica Italiana}.
\newblock Springer, [Cham]; Unione Matematica Italiana, Bologna, 2016.

\bibitem{CF}
C.~Carrillo and P.~Fife.
\newblock Spatial effects in discrete generation population models.
\newblock {\em Journal of Mathematical Biology}, 50(2):161--188, 2005.

\bibitem{Dac}
B.~Dacorogna.
\newblock {\em Direct methods in the calculus of variations}, volume~78 of {\em
  Applied Mathematical Sciences}.
\newblock Springer-Verlag, Berlin, 1989.

\bibitem{DGLz2013}
Q.~Du, M.~Gunzburger, R.~B. Lehoucq, and K.~Zhou.
\newblock A nonlocal vector calculus, nonlocal volume-constrained problems, and
  nonlocal balance laws.
\newblock {\em Mathematical Models and Methods in Applied Sciences},
  23(03):493--540, 2013.

\bibitem{Elbau}
P.~{Elbau}.
\newblock {Sequential Lower Semi-Continuity of Non-Local Functionals}.
\newblock {\em ArXiv e-prints}, Apr. 2011.

\bibitem{FG}
M.~Foss and J.~Geisbauer.
\newblock Higher differentiability in the context of {B}esov spaces for a class
  of nonlocal functionals.
\newblock {\em Evol. Equ. Control Theory}, 2(2):301--318, 2013.

\bibitem{FR}
M.~Foss and P.~Radu.
\newblock Differentiability and integrability properties for solutions to
  nonlocal equations.
\newblock In {\em New Trends in Differential Equations, Control Theory and
  Optimization: Proceedings of the 8th Congress of Romanian Mathematicians},
  page 105. World Scientific, 2016.

\bibitem{HR}
B.~Hinds and P.~Radu.
\newblock Dirichlet's principle and wellposedness of solutions for a nonlocal
  {$p$}-{L}aplacian system.
\newblock {\em Appl. Math. Comput.}, 219(4):1411--1419, 2012.

\bibitem{leoni}
G.~Leoni and M.~Morini.
\newblock Necessary and sufficient conditions for the chain rule in wloc1, 1
  (rn; rd) and bvloc (rn; rd).
\newblock {\em Journal of the European Mathematical Society}, 9(2):219--252,
  2007.

\bibitem{Mizel}
M.~Marcus and V.~J. Mizel.
\newblock Absolute continuity on tracks and mappings of sobolev spaces.
\newblock {\em Archive for Rational Mechanics and Analysis}, 45(4):294--320,
  1972.

\bibitem{DuMengesha}
T.~Mengesha and Q.~Du.
\newblock The bond-based peridynamic system with {D}irichlet-type volume
  constraint.
\newblock {\em Proc. Roy. Soc. Edinburgh Sect. A}, 144(1):161--186, 2014.

\bibitem{MCoV}
C.~B. Morrey, Jr.
\newblock {\em Multiple integrals in the calculus of variations}.
\newblock Die Grundlehren der mathematischen Wissenschaften, Band 130.
  Springer-Verlag New York, Inc., New York, 1966.

\bibitem{Pedregal1997}
P.~Pedregal.
\newblock Nonlocal variational principles.
\newblock {\em Nonlinear Analysis: Theory, Methods \& Applications},
  29(12):1379--1392, 1997.

\bibitem{Silling2000}
S.~Silling.
\newblock {Reformulation of elasticity theory for discontinuities and
  long-range forces}.
\newblock {\em Journal of the Mechanics and Physics of Solids}, 48:175--209,
  2000.

\end{thebibliography}
\end{document}